\documentclass[11pt,reqno,a4paper]{amsart}
\usepackage{latexsym}
\usepackage{amsmath}
\usepackage{amssymb}
\usepackage{amsthm}
\usepackage{amscd}
\usepackage{graphicx}
\usepackage{xcolor}
\usepackage[colorlinks,citecolor=blue]{hyperref}
\usepackage{tikz}
\usepackage{tikz-cd}
\usepackage{caption} 
\usepackage{bm}
\usepackage[a4paper,top=3cm,bottom=3cm,left=3cm,right=3cm]{geometry}
\usepackage{cite}

\newtheorem{theorem}{Theorem}[section]
\newtheorem*{theorem*}{Theorem}
\newtheorem{corollary}[theorem]{Corollary}
\newtheorem{lemma}[theorem]{Lemma}
\newtheorem{proposition}[theorem]{Proposition}

\newtheorem{quest}[theorem]{Question}
\theoremstyle{definition}
\newtheorem{definition}[theorem]{Definition}
\newtheorem{ex}[theorem]{Example}
\theoremstyle{remark}

\newtheorem*{claim}{Claim}

\newcommand{\Ku}{\mathcal{K}u}
\numberwithin{equation}{section}
\newcommand{\bv}{\mathbf{v}}
\newcommand{\bw}{\mathbf{w}}

\DeclareMathOperator{\oh}{\mathcal{O}}

\begin{document}

\title[]{Linear subspaces of the intersection of two quadrics via Kuznetsov component}

\author{Yanjie Li}
\address{Sun-Yat Sen University, School of Mathematics(Zhuhai), China}
\email{liyj293@mail2.sysu.edu.cn}
\author{Shizhuo Zhang}
\address{Institut de Mathématiqes de Toulouse, UMR 5219, Université de Toulouse, Université Paul Sabatier, 118 route de
Narbonne, 31062 Toulouse Cedex 9, France}
\email{shizhuozhang@mpim-bonn.mpg.de,shizhuo.zhang@math.univ-toulouse.fr}


\begin{abstract}
   Let $Q_i(i=1,2)$ be $2g$ dimensional quadrics in $\mathbb{P}^{2g+1}$ and let $Y$ be the smooth intersection $Q_1\cap Q_2$. We associate the linear subspace in $Y$ with vector bundles on the hyperelliptic curve $C$ of genus $g$ by the left adjoint functor of $\Phi:D^b(C)\rightarrow D^b(Y)$. As an application, we give a different proof of the classification of line bundles and stable bundles of rank $2$ on hyperelliptic curves given by Desale and Ramanan. When $g=3$, we show that the projection functor induces a closed embedding $\alpha:Y\rightarrow SU^s_C(4,h)$ into the moduli space of stable bundles on $C$ of rank $4$ of fixed determinant.
\end{abstract}

\maketitle

\setcounter{tocdepth}{1}
\tableofcontents

\section{Introduction}

Let $Y$ be a smooth del Pezzo threefold of degree $4$, it is the intersection of two quadrics $Y:=Q_1\cap Q_2$ in $\mathbb{P}^5$. Consider the pencil of quadrics $\{Q_\lambda\}_{\lambda\in\mathbb{P}^1}$ generated by $Q_1$ and $Q_2$. If $Y$ is smooth, then generic $Q_{\lambda}$ is smooth and there are precisely 6 distinct points $\lambda_1,\ldots,\lambda_6\in\mathbb{P}^1$ for which the quadric $Q_\lambda$ is degenerate. Consider the double covering $C\rightarrow\mathbb{P}^1$ with ramification points $\lambda_1,\ldots,\lambda_6$, then we get a smooth hyperelliptic curve $C$. By \cite[Theorem 1.1]{newstead1968stable}, there is a moduli space interpretation of $Y:$ it is isomorphic to the moduli space of stable rank two vector bundles over $C$ of fixed determinant of odd degree. Now we consider semi-orthogonal decomposition of $D^b(Y):$
$$D^b(Y)=\langle\Ku(Y),\oh_Y,\oh_Y(1)\rangle,$$
where $\Ku(Y)$ as the right orthogonal complement of the line bundles $\oh_Y,\oh_Y(1)$ is called Kuznetsov component. It is shown in \cite{bondal1995semiorthogonal} that $\Ku(Y)\simeq D^b(C)$ and the embedding $\Phi:\Ku(Y)\hookrightarrow D^b(Y)$ is given by the Fourier-Mukai functor $\phi_{S}$, where $S$ as the Fourier-Mukai kernel is given by the universal spinor bundle on $C\times Y$. A simple computation shows that the numerical Grothendieck group $\mathcal{N}(\Ku(Y))\cong\mathbb{Z}^2=\langle\bv,\bw\rangle$ is a rank two lattice generated by two vectors $\bv$ and $\bw$. Denote by $\sigma$ the unique stability condition(up to $\widetilde{GL}^+(2,\mathbb{R})$-action).  Let $\mathrm{pr}=\mathbb{L}_{\oh_Y}\mathbb{L}_{\oh_Y(1)}$ be the projection functor $D^b(Y)\rightarrow\Ku(Y)$ which induces a closed immersion of $Y$ into the Bridgeland moduli space $\mathcal{M}_{\sigma}(\Ku(Y),\bw),$ 
and it is shown in \cite[Section 5.2]{rota} that $$Y\cong\{E\in\mathcal{M}_{\sigma}(\Ku(Y),w)|\mathrm{Hom}(E,\Phi^!(\oh_Y))=k^5\}\cong\{E\in\mathcal{M}_C(2,1)|\mathrm{hom}(E,\mathcal{R}[1])=5\},$$

where $\mathcal{R}=\Phi^!(\oh_Y)[-1]$ is a \emph{second Raynaud bundle}. Furthermore, it is shown in \cite[Section 6.1]{feyzbakhsh23torelli} that fixing determinant is equivalent to imposing \emph{Brill-Noether} condition. Thus the natural question is that if intersection of quadrics in higher dimensional projective spaces admits the moduli space interpretation. On the other hand, using the moduli space reconstruction for $Y$, one can identify the cotangent bundle $T^*Y$ of $Y$ with moduli space $\mathcal{M}_{\mathrm{Higgs}}$ of Higgs bundles, that is points $(E,\phi)$ with $E\in\mathcal{M}(2,\mathcal{L})$ and $\phi:E\rightarrow E\otimes\omega_C$ a homomorphism with $\mathrm{Tr}\phi=0$. The \emph{Hitchin map} $T^*\mathcal{M}_{\mathrm{Higgs}}\xrightarrow{\mathrm{det}(\phi)} H^0(\omega^{\otimes 2})$ is a Lagrangian fibration. In the work \cite{beauville2023symmetric}, the authors show the intersection of two quadrics in higher dimensional projective space also admits a Lagrangian fibration. Thus it would be very interesting to find out the moduli interpretation for general intersection of two quadrics, as the first step to understand the Lagrangian fibration in an alternative perspective. These constitute the first motivation of our work.

Furthermore, it is known that the intermediate Jacobian $J(Y)$ is isomorphic to $J(C)$, consisting of degree $0$ line bundles over $C$. Moreover $J(C)$ can be identified with Hilbert scheme of lines on $Y$. It is natural to ask if the Hilbert scheme of linear subspaces of intersection of two quadrics in higher projective space can be identified with moduli space of stable vector bundles over the associated hyperelliptic curve. Attempts to answer this question is the second motivation of our work. 

\subsection{Main Results}
Let $Q_1:\sum_{j=1}^{2g+1}x_j^2=0,~Q_2:\sum_{j=1}^{2g+1}\lambda_jx_j^2=0$ be two quadrics in $\mathbb{P}^{2g+1}$, where $\lambda_j(1\leq j \leq 2g+1)$ are distinct complex numbers. Let $Y$ be the smooth complete intersection $Q_1\cap Q_2$ in $\mathbb{P}^{2g+1}$. let $C$ be double cover of $\mathbb{P}^1$ ramified at $\lambda_j(1\leq j \leq 2g+1)$, which is a hyperelliptic curve of genus $g$. It can be identfied with the fine moduli space of the spinor bundles on $Y$. The vector bundles on $C$ are closely related to linear subspaces in $Y$: it is shown in \cite[Theorem 4.8]{reid1972complete} that the Jacobian of $C$ is isomorphic to the variety of $(g-1)$ dimensional subspaces in $Y$. In \cite[Theorem 1]{desale1976classification}, the authors showed that the moduli space of stable vector bundles on $C$ of rank $2$ and fixed determinant of odd degree is isomorphic to the variety of $(g-2)$ dimensional subspaces in $Y$. Then in \cite{ramanan1981orthogonal}, the author also gave a description of the variety of $(g-n)$ dimensional subspaces in $Y$ as the moduli space of orthogonal bundles of rank $2n$ with some additional assumptions. 

The universal spinor bundle $S$ on $C\times Y$ induces a full and faithful embedding $\Phi:D^b(C)\rightarrow D^b(Y)$, the semi-orthogonal decomposition of $D^b(Y)$ is given by 
$$D^b(Y)=\langle D^b(C),\oh_Y,\ldots,\oh_Y(2g-3)\rangle.$$
The first main result of this paper is to give a categorical description of the relation between linear subspaces in $Y$ and vector bundles on $C$, we have 
	
\begin{theorem}[Propositions \ref{lin sp in Y}, \ref{auto fr m to m+1}, \ref{ext prop} and Corollary \ref{cor of compare}]\label{thm_first_result}
    Let $V$ be a linear subspace in $Y$ of dimension $l$. Denote the left adjoint functor of $\Phi$ by $\Phi^*$ and the involution on $C$ by $\tau:C\rightarrow C$.\\
    $(1)$ For $2g-3-l\leq m \leq 2g-3$, $\mathcal{F}_{m,V}:=\Phi^*(\oh_V(m))[-m-2]$ is a vector bundle on $C$ of rank $2^{g-1-l}$.\\ 
    $(2)$ There's a line bundle $\mathcal{L}$ on $C$ of degree $-1$ (depends on the universal spinor $S$), satisfying $\mathcal{F}_{m+1,V}\cong \tau^*\mathcal{F}_{m,V}\otimes \mathcal{L}~(2g-3-l\leq m \leq 2g-4)$. In particular, $\deg(\mathcal{F}_{m,V})-\deg(\mathcal{F}_{m+1,V})=2^{g-1-l}$.\\
    $(3)$ If there's a $(l-1)$ dimensional subspace $L$ with $L\subset V$, we have a non-trivial extension
    $$0\rightarrow \mathcal{F}_{m,V} \rightarrow \mathcal{F}_{m,L} \rightarrow \mathcal{F}_{m-1,V}\rightarrow 0~(2g-2-l \leq m\leq 2g-3).$$
    $(4)$ For two linear subspaces $V_1,V_2$, we have the isomorphism $\mathrm{Hom}(\oh_{V_1},\oh_{V_2})\cong \mathrm{Hom}(\mathcal{F}_{m,V_1},\mathcal{F}_{m,V_2})$. In particular, the map $V\mapsto [\mathcal{F}_{m,V}]$ is injective, and $V_2\subseteq V_1$ iff $\mathrm{Hom}(\mathcal{F}_{m,V_1},\mathcal{F}_{m,V_2})\ne 0$.
\end{theorem}

Let $\mathcal{H}_l$ be the variety of $l$ dimensional subspaces in $Y$. We show that the map $\alpha^l_m:V\mapsto \mathcal{F}_{m,V}$ is a closed embedding from $\mathcal{H}_l$ to the moduli space of stable vector bundles when the associated bundles $\mathcal{F}_{m,V}$ are stable. As an application, we give an alternative proof of \cite[Theorem 4.8]{reid1972complete} and \cite[Theorem 2]{desale1976classification} using derived category. 
When $g=3$, we can also realize the $5$ dimensional intersection $Y$ as a closed subvariety of the moduli space $SU_C^{s}(4,h)$ on the genus $3$ hyperelliptic curve $C$. All of above constitute the second main result of our paper. 

\begin{theorem}[Proposition \ref{stable then closed embedding}, Theorem \ref{pl iso to pic}, Theorem \ref{stab bd for line}]\label{thm_second_result}
    Denote by $U^s_C(r,d)$ the moduli space of stable bundles on $C$ of rank $r$ and degree $d$ and by $SU^s_C(r,\xi)$ the moduli space of stable bundles on $C$ of rank $r$ and fixed determinant $\xi$. We set $d_{m,l}:=\deg(\mathcal{F}_{m,l})$.\\
    $(1)$ If the vector bundle $\mathcal{F}_{m,V}$ is stable for each linear subspace $V$ in $Y$ of dimension $l$, then we have a closed embedding $\alpha^l_m:\mathcal{H}_l\rightarrow U_C^s(2^{g-1-l},d_{m,l}),~[V]\mapsto [\mathcal{F}_{m,V}]$. In particular, we have isomorphisms $\alpha^{g-1}_m:\mathcal{H}_{g-1}\rightarrow \mathrm{Pic}^{d_{m,g-1}}(C)~(g-2\leq m \leq 2g-3)$.\\
    $(2)$ If $V$ is of dimension $g-2$, then the associated rank $2$ bundle $\mathcal{F}_{m,V}$ is stable and of fixed determinant of odd degree. We have isomorphisms $\alpha^{g-2}_m:\mathcal{H}_{g-2}\rightarrow SU_C(2,h_m)~(g-1\leq m \leq 2g-3)$, where $h_m=\det(\mathcal{F}_{m,V})$ is a fixed line bundle of odd degree on $C$ for each $m$.\\
    $(3)$ If $g$ is $3$, then for every point $p\in Y$, the rank $4$ bundle $\mathcal{F}_{3,p}$ is stable of fixed determinant. We have a closed embedding $\alpha^0:Y\rightarrow SU^s_C(4,h),~p\mapsto \mathcal{F}_{3,p}$, where $h=\det(\mathcal{F}_{3,p})$ is a fixed line bundle on $C$ with $\deg(h)\equiv 0~(\mathrm{mod}~4)$. 
\end{theorem}

\subsection{Organization of the article}
In Section~\ref{section_linear_subspaces_projection_functors} we compute the image of structure sheaves of linear subspaces of intersection of two quadrics in projective spaces under projection functors and we prove Theorem~\ref{thm_first_result}. In Section~\ref{section_projection_induce_morphism}, we show the projection functor induces a closed embedding of Hilbert scheme of $l$-dimensional linear subspaces of intersection of two quadrics into moduli space of stable vector bundles over some curve under mild assumption. In particular, we prove Theorem~\ref{thm_second_result}. In Section~\ref{section_seek_BN_conditions}, we make the first attempt to find out the Brill-Noether condition for the image of the projection functor inside the moduli space. 

\subsection{Acknowledgements}
The first author would like to thank Duo Li for useful discussions. The second author would like to thank Arend Bayer, Jie Liu, Zhiyu Liu and Hao Sun for useful discussion. He thanks Daniele Faenzi, Laurent Manivel and Claire Voisin for their interest in this work. The second author is supported by ANR project FanoHK, grant ANR-20-CE40-0023, Deutsche Forschungsgemeinschaft under Germany's Excellence Strategy-EXC-2047$/$1-390685813. Part of the work is finished when the second author visit Sun-Yat Sen university(Zhuhai), Max-Planck institute for mathematics and Hausdorff research institute for mathematics. He is grateful for their excellent hospitality and support.

\textbf{Notations}: We assume all schemes are over $\mathbb{C}$. Let $Y$ be a smooth complete intersection of two quadrics $Q_1\cap Q_2$ in $\mathbb{P}^{2g+1}(g\geq 2)$. 
Let $C$ be hyperelliptic curve associated with $Y$, which is a double cover of $\mathbb{P}^1$ ramified at $2g+1$ critical values of genus $g$. Let $S$ be the universal spinor bundle on $C\times Y$.
Denote the full and faithful embedding functor
 given by the Fourier mukai transform with kernel $S$ by $\Phi:=\phi_S:D^b(C)\rightarrow D^b(Y)$ and the left and right adjoint of $\Phi$ by $\Phi^{*}$ and $\Phi^{!}$. Let $\mathcal{H}_l$ be the Hilbert Scheme of linear subspace of dimension $l$ in $Y$. Denote by $U^s_C(r,d)$ the moduli space of stable bundles on $C$ of rank $r$ and degree $d$ and by $SU^s_C(r,\xi)$ the moduli space of stable bundles on $C$ of rank $r$ and fixed determinant $\xi$.

\section{Linear subspaces in the Intersection of two quadrics}
\label{section_linear_subspaces_projection_functors}
In this section we associate linear subspaces in $Y$ with vector bundles on $C$ via projecting the structure sheaves of linear subspaces by adjoint functors.

\subsection{Projection of the linear subspaces}
Recall that the maximal linear subspaces of $Y$ is of dimension $g-1$. We're going to associate each linear subspace $V$ not only one but a sequence of vector bundles $\mathcal{F}_{m,V}$ indexed by $m$. 
\begin{proposition}\label{lin sp in Y}
Let $V(\subset Y)$ be a linear subspace in $Y$ of dimension $l~(l\leq g-1)$. For $2g-3-l\leq m \leq 2g-3$, $\mathcal{F}_{m,V}:=\Phi^{*}(\oh_V(m))[-m-2]$ is a vector bundle of rank $2^{g-1-l}$.
\end{proposition}
\begin{proof}
    For each $x\in C$, let $S_x$ and $Q_x$ be the corresponding spinor bundle and quadric, we have
    \begin{equation*}
    \begin{aligned}
         \mathrm{Ext}^{\bullet}(\mathcal{F}_{m,V},\oh_x)
          &\cong \mathrm{Ext}^{\bullet}(\oh_V(m)[-m-2],S_x)~(\text{by adjunction})~\\
          &\cong \mathrm{Ext}^{2g-3-m-\bullet}(S_x, \oh_V(m-2g+2))^{\vee}.~(\text{by Serre duality on }Y)~
    \end{aligned}
\end{equation*}
The quadric $Q_x$ is either a smooth quadric or the cone over a smooth quadric of dimension $2g-1$ and $Y$ does not meet the singular point of the cone. If $Q_x$ is smooth, take a $2l+2$ dimensional linear space $M\subset \mathbb{P}^{2g+1}$ containing $V$ that is not tangent to $Q_x$, we denote $Q_{2l+1}:=Q_x\cap M$ the smooth quadric of dimension $2l+1$ satisfying $V\subset Q_{2l+1}\subset Q_x$. Let $S_{2l+1}$ be the spinor bundle on $Q_{2l+1}$. By \cite[Theorem 1.4]{ottaviani1988spinor}, we have the isomorphism $S_x|_{Q_{2l+1}}\cong S_{2l+1}^{\oplus 2^{g-l-1}}$. If $Q_x$ is degenerate, we take a hyperplane $H\subset \mathbb{P}^{2g+1}$ containing $V$ that is not tangent to $Q_x$ and set $Q_{2g-1}:=Q_x\cap H$. $Q_x$ is the cone over $Q_{2g-1}$ and $S_x$ is the pull back of the spinor bundle $S_{2g-1}$ on $Q_{2g-1}$. We can also find a $2l+1$ dimensional smooth quadric $Q_{2l+1}$ with $V\subset Q_{2l+1} \subseteq Q_{2g-1}$. By \cite[Theorem 1.4]{ottaviani1988spinor} we also have $S_x|_{Q_{2l+1}}\cong S_{2g-1}|_{Q_{2l+1}}\cong S_{2l+1}^{\oplus 2^{g-l-1}}$. By \cite[Theorem 2.5]{ottaviani1988spinor}, we have isomorphisms $$S_x|_V\cong ({S_{2l+1}|_V})^{\oplus 2^{g-l-1}}\cong (\bigoplus_{i=0}^{l}\Omega_{\mathbb{P}^{l}}^{i}(i))^{\oplus 2^{g-l-1}}. $$ 
If we denote the summand by $E:=\bigoplus_{i=0}^{l}\Omega_{\mathbb{P}^{l}}^{i}(i)$, we have
$$(S_x|_V)^{\vee}\cong(\bigoplus_{i=0}^{l}\Omega_{\mathbb{P}^{l}}^{l-i}(l-i+1))^{\oplus 2^{g-l-1}}\cong(\bigoplus_{j=0}^{l}\Omega_{\mathbb{P}^{l}}^{j}(j+1))^{\oplus 2^{g-l-1}}=E(1)^{\oplus 2^{g-l-1}}.$$ Therefore by \cite[Lemma 2.4]{ottaviani1988spinor}, note that $-l\leq m-2g-3\leq 0$, we have isomorphisms 
\begin{equation*}
    \begin{aligned}
        \mathrm{Ext}^{\bullet}(\mathcal{F}_{m,V},\oh_x)^{\vee} &\cong \mathrm{Ext}^{2g-3-m-\bullet}(S_x, \oh_V(m-2g+2))\\
        &\cong H^{2g-3-m-\bullet}(\mathbb{P}^{l}, E(m-2g+3))^{\oplus 2^{g-l-1}}\cong \mathbb{C}^{\oplus 2^{g-l-1}},
    \end{aligned}
\end{equation*}
which implies that $\mathcal{F}_{m,V}$ is a vector bundle of rank $2^{g-l-1}$. 
\end{proof}
For the sake of convenience, when we write the symbol $\mathcal{F}_{m,V}$ for a linear subspace $V$ of dimension $l$, we assume that $2g-3-l\leq m \leq 2g-3$. In particular if we take $V$ as the maximal subspace of $Y$, we have the following Corollary as a categorical description of \cite[Theorem 4.8]{reid1972complete}. 
\begin{corollary}
    Let $V\subset Y$ is a linear subspace of maximal dimension $g-1$. For $g-2\leq m \leq 2g-3$, $\mathcal{F}_{m,V}=\Phi^{*}(\mathcal{\oh}_V(m))[-m-2]$ is a line bundle.
\end{corollary}

The vector bundles $\mathcal{F}_{m,V}$ and $\mathcal{F}_{m+1,V}$ only differ by an autoequivalence of $D^b(C)$.
\begin{proposition}\label{auto fr m to m+1}
    The rotation functor $\mathbf{O}:=\Phi^*\circ(-\otimes \oh_Y(1)[-1])\circ \Phi$ is an autoequivalence of $D^b(C)$. We have $\mathbf{O}(-)=\tau^*(-)\otimes \mathcal{L}$ for a line bundle $\mathcal{L}$ on $C$ of degree $-1$ with $(\mathcal{L}\otimes \tau^*\mathcal{L})^{\otimes g-1}\cong \omega_C^{-1}$. Let $V$ be a linear subspace in $Y$ of dimension $l$, we have $\mathcal{F}_{m+1,V}\cong \mathbf{O}(\mathcal{F}_{m,V}) \cong \tau^*\mathcal{F}_{m,V}\otimes \mathcal{L}$ for $2g-3-l \leq m \leq 2g-4$.
\end{proposition}
\begin{proof}
    Note that we have $\Phi \circ \mathbf{O} \circ \Phi^{-1}=\mathbb{L}_{\oh_Y}(-\otimes \oh_Y(1)[-1])$. As in \cite[Lemma 4.1]{kuznetsov2003derived}, we iterate the functor $\mathbf{O}$ and get
    $$\Phi\circ \mathbf{O}^{2g-2} \circ \Phi^{-1}=\mathbb{L}_{\oh_Y}\mathbb{L}_{\oh_Y(1)}\cdots \mathbb{L}_{\oh_Y(2g-3)}\circ (-\otimes \oh_Y(2g-2))[2-2g]=\Phi\circ \Phi^*\circ \mathbb{S}_{Y}^{-1}[1]=\Phi\circ \mathbb{S}_C^{-1}[1],$$
    and hence $\mathbf{O}$ is an autoequivalence.
    Since $C$ is of general type, any autoequivalence of $D^b(C)$ is a composition of the pullback of an automorphism of $C$ and a twist of line bundle by \cite[Theorem 3.1]{bondal2001reconstruction}.  For any closed point $x\in C$, we have $\mathbf{O}(\oh_x)=\mathbb{L}_{\oh_Y}(S_x(1))[-1]\cong S_{\tau(x)}$ by \cite[Theorem 2.8 (ii)]{ottaviani1988spinor}. Therefore the autoequivalence $\mathbf{O}$ is given by $\mathbf{O}(\mathcal{F})\cong \tau^*\mathcal{F}\otimes \mathcal{L}$ for some line bundle $\mathcal{L}$ on $C$. We have $\mathcal{F}\otimes (\mathcal{L}\otimes \tau^*\mathcal{L})^{\otimes g-1}=\mathbf{O}^{2g-2}(\mathcal{F})=\Phi^*\circ \mathbb{S}_Y^{-1}(\mathcal{F})[1]=\mathbb{S}_{C}^{-1}(\mathcal{F})[1]=\mathcal{F}\otimes \omega_C$ for any $\mathcal{F}\in D^b(C)$, implying $(\mathcal{L}\otimes \tau^*\mathcal{L})^{\otimes g-1}\cong \omega_C^{-1}$ and $\deg(\mathcal{L})=-1$.

    For $2g-3-l\leq m \leq 2g-2$, we first show $\oh_V(m)\in \langle \oh_Y(m+1),\dots,\oh_Y(2g-3) \rangle^{\perp}$. For $m<n\leq 2g-3$, we have $\mathrm{Ext}^i(\oh_Y(n),\oh_V(m))\cong H^i(V,\oh_V(m-n))=0$ from the inequalities $-l\leq m-2g+3\leq m-n<0$. 
    Therefore we have $\Phi(\mathcal{F}_{m,V})=\mathbb{L}_{\oh_Y}\cdots \mathbb{L}_{\oh_Y(m)}(\oh_V(m))[-m-2]$. We conclude that
    $$\Phi \circ \mathbf{O}(\mathcal{F}_{m,V})=\mathbb{L}_{\oh_Y}(\Phi(\mathcal{F}_{m,V})\otimes \oh_Y(1))[-1]=\mathbb{L}_{\oh_Y}(\mathbb{L}_{\oh_Y}\cdots \mathbb{L}_{\oh_Y(m)}(\oh_V(m))\otimes \oh_Y(1))[-m-3]$$
    $$\cong \mathbb{L}_{\oh_Y}\mathbb{L}_{\oh_Y(1)}\cdots \mathbb{L}_{\oh_Y(m+1)}\oh_V(m+1)[-m-3]=\Phi(\mathcal{F}_{m+1,V})$$
    and $\mathcal{F}_{m+1,V}\cong \mathbf{O}(\mathcal{F}_{m,V}) \cong \tau^*\mathcal{F}_{m,V}\otimes \mathcal{L}$. 
\end{proof}

Now we define the map from the variety of linear subspaces to the set of vector bundles.
\begin{definition}\label{proj of lin}
    We denote by $\alpha_m^l:=V\mapsto [\mathcal{F}_{m,V}]$ the map that maps a linear subspace $V$ in $Y$ of dimension $l$ to the isomorphism class of the associated bundle $\mathcal{F}_{m,V}$.
\end{definition}

The inclusion relationship between linear spaces can be interpreted as the extension of the associated vector bundles, more precisely we have

\begin{proposition}\label{ext prop}
    Let $V\subset Y$ be a linear subspace of dimension $l$ and let $L\subset V$ be a linear subspace of dimension $l-1$, for $2g-2-l \leq m \leq 2g-3$, there exists a non-trivial extension
    $$0\rightarrow \mathcal{F}_{m,V} \rightarrow 
 \mathcal{F}_{m,L} \rightarrow \mathcal{F}_{m-1,V}\rightarrow 0.$$
\end{proposition}
\begin{proof}
Consider $L$ as a hyperplane section of $V$, we have a short exact sequence
$$0\rightarrow \oh_V(-1) \rightarrow \oh_V \rightarrow \oh_L \rightarrow 0. $$
After twisting, we further apply the left adjoint functor and shifts to get an exact triangle
$$\Phi^{*}(\oh_V(m-1))[-m-2]\rightarrow\Phi^{*}(\oh_V(m))[-m-2]\rightarrow\Phi^{*}(\oh_L(m))[-m-2], $$
 we get the extension by taking sheaf cohomology. 

To prove this extension is non-trivial, we show  $\mathrm{Hom}(\mathcal{F}_{m,L},\mathcal{F}_{m,L})\cong \mathrm{Hom}(\oh_L,\oh_L)=\mathbb{C}$ by the Lemma \ref{compare hom} below.
\end{proof}

\begin{lemma}\label{compare hom}
    For two linear subspaces $V_1,V_2$, we have the isomorphism 
    $$\mathrm{Hom}(\mathcal{F}_{m,V_1},\mathcal{F}_{m,V_2})\cong \mathrm{Hom}(\oh_{V_1},\oh_{V_2}).$$
    If $\mathrm{dim}(V_1)=\mathrm{dim}(V_2)+d$ and $d>0$, we have isomorphism
    $$\mathrm{Hom}(\mathcal{F}_{n,V_2},\mathcal{F}_{n-d,V_1})\cong \mathrm{Ext}^{d}(\oh_{V_2},\oh_{V_1}(-d)) \cong \mathrm{Ext}^{d-1}(\mathcal{I}_{V_2},\oh_{V_1}(-d)),$$
    and for any linear subspace $V$, we have the injections
    $$\mathrm{Hom}(\mathcal{I}_V,\oh_V)\hookrightarrow \mathrm{Ext}^{1}(\mathcal{I}_V,\mathcal{I}_V)\hookrightarrow\mathrm{Ext}^{1}(\mathcal{F}_{m,V},\mathcal{F}_{m,V}).$$
\end{lemma}
\begin{proof}
    For the first statement, we assume that $\mathrm{dim}(V_1)=l_1,\mathrm{dim}(V_2)=l_2$. Set $$L_j=\mathbb{L}_{\oh_Y(m-j)}\cdots \mathbb{L}_{\oh_Y(m)}\oh_{V_2}(m)~(0\leq j \leq m).$$ Since $m$ is at least $2g-3-l_2$, we have $\oh_{V_2}(m)\in \langle \oh_Y(m+1),\dots,\oh_Y(2g-3)\rangle^{\perp}$ and $L_m=\Phi\Phi^{*}(\oh_{V_2}(m))$. We also make the convention $L_{-1}=\oh_{V_2}(m)$. The assertion $\mathrm{Hom}(\mathcal{F}_{m.V_1},\mathcal{F}_{m,V_2})\cong \mathrm{Hom}(\oh_{V_1},\oh_{V_2})$ is equivalent to $\mathrm{Hom}(\oh_{V_1}(m),\Phi\Phi^*(\oh_{V_2}(m)))\cong \mathrm{Hom}(\oh_{V_1}(m),\oh_{V_2}(m))$, it suffices to show that we have isomorphisms $\mathrm{Hom}(\oh_{V_1}(m),L_{j-1})\cong \mathrm{Hom}(\oh_{V_1}(m),L_j)~(0\leq j \leq m)$. 
    
    The left mutation is defined by $\bigoplus_{i}\mathrm{Hom}(\oh_Y(m-j)[i],L_{j-1})\otimes \oh_Y(m-j)[i]\rightarrow L_{j-1} \rightarrow L_j$. By induction on $j$ we can show that for any coherent sheaf $\mathcal{G}$, $\mathrm{Hom}(\mathcal{G}[n],L_j)$ vanishes for $n>j+1$. So indeed we have
    $\bigoplus_{i\leq j}\mathrm{Hom}(\oh_Y(m-j)[i],L_{j-1})\otimes \oh_Y(m-j)[i]\rightarrow L_{j-1} \rightarrow L_j.$
    Now we show $\mathrm{Hom}(\oh_{V_1}(m),\oh_Y(m-j)[i])=0$ for $i\leq j$. Note the isomorphisms 
    \begin{equation*}
        \begin{aligned}
            \mathrm{Hom}(\oh_{V_1}(m),\oh_Y(m-j)[i])&\cong \mathrm{Hom}(\oh_Y(m-j),\oh_{V_1}(2-2g+m)[2g-1-i])^{\vee}\\
            &\cong H^{2g-1-i}(V_1,\oh_{V_1}(j-2g+2))^{\vee}.
        \end{aligned}
    \end{equation*}
    Similarly $\mathrm{Hom}(\oh_{V_1}(m),\oh_Y(m-j)[i+1])\cong H^{2g-2-i}(V_1,\oh_{V_1}(j-2g+2))^{\vee}$. Note the inequalities $2g-2-i\geq 2g-2-m>0$. If one of the cohomology groups does not vanish, $2g-1-i$ or $2g-2-i$ must equal $l_1$, that is, $i\geq 2g-2-l_1$. But it implies $j-2g+2\geq i-2g-2 \geq -l_1$, in which case the cohomology groups are zero. Therefore we have $ \mathrm{Hom}(\oh_{V_1}(m),\oh_Y(m-j)[i])=\mathrm{Hom}(\oh_{V_1}(m),\oh_Y(m-j)[i+1])=0$ for $i\leq j$. Applying $\mathrm{Hom}(\oh_{V_1}(m),-)$ to the above exact triangle, we have $\mathrm{Hom}(\oh_{V_1}(m),L_{j-1})\cong\mathrm{Hom}(\oh_{V_1}(m),L_j)$ and $\mathrm{Hom}(\mathcal{F}_{m,V_1},\mathcal{F}_{m,V_2})\cong \mathrm{Hom}(\oh_{V_1},\oh_{V_2})$. The other assertions are proved similarly and we give a detailed proof for the last statement.

    Let $V$ be a subspace in $Y$ of dimension $l$. To show $\mathrm{Hom}(\mathcal{I}_V,\oh_V)\hookrightarrow \mathrm{Ext}^1(\mathcal{I}_V,\mathcal{I}_V)$, we apply $\mathrm{Hom}(-,\oh_Y)$ to $0\rightarrow \mathcal{I}_V\rightarrow \oh_Y \rightarrow \oh_V\rightarrow 0$, and get $\mathrm{Hom}(\mathcal{I}_V,\oh_Y)\cong \mathrm{Hom}(\oh_Y,\oh_Y)$. Then we apply $\mathrm{Hom}(\mathcal{I}_V,-)$ to get $\mathrm{Hom}(\mathcal{I}_V,\oh_V)\hookrightarrow \mathrm{Ext}^1(\mathcal{I}_V,\mathcal{I}_V)$.
    
    The morphism $\mathrm{Ext}^{1}(\mathcal{I}_V,\mathcal{I}_V)\rightarrow\mathrm{Ext}^{1}(\mathcal{F}_{m,V},\mathcal{F}_{m,V})$ in the last statement is given by 
    \begin{equation*}
        \begin{aligned}
            \mathrm{Ext}^{1}(\mathcal{I}_V,\mathcal{I}_V)\cong \mathrm{Ext}^{1}(\mathcal{I}_V(m),\mathcal{I}_V(m))\rightarrow &\mathrm{Ext}^{1}(\mathcal{I}_V(m),\Phi\Phi^{*}(\mathcal{I}_V(m)))\\
            \cong & \mathrm{Ext}^{1}(\mathcal{I}_V(m),\Phi\Phi^{*}(\oh_V(m))[-1])\cong \mathrm{Hom}(\mathcal{I}_V(m),\Phi\Phi^{*}(\oh_V(m))).
        \end{aligned}
    \end{equation*}
    We set $L_j=\mathbb{L}_{\oh_Y(m-j)}\cdots\mathbb{L}_{\oh_Y(m)}\oh_V(m)$ with $L_{-1}=\oh_V(m)$. The morphism can be identified as
    $\mathrm{Hom}(\mathcal{I}_V(m),L_0)\rightarrow \mathrm{Hom}(\mathcal{I}_V(m),L_m)$. We show the injections $\mathrm{Hom}(\mathcal{I}_V(m),L_{j-1})\hookrightarrow \mathrm{Hom}(\mathcal{I}_V(m),L_j)(1\leq j \leq m)$. 

    The left mutation is defined by
    $\bigoplus_{i\leq j}\mathrm{Hom}(\oh_Y(m-j)[i],L_{j-1})\otimes \oh_Y(m-j)[i]\rightarrow L_{j-1} \rightarrow L_j.$
    It suffices to show $\mathrm{Hom}(\mathcal{I}_V(m),\oh_Y(m-j)[i])=0$ for $1\leq j \leq m,~i\leq j,~2g-3-l\leq m \leq 2g-3$. By Serre duality, we have $\mathrm{Hom}(\mathcal{I}_V(m),\oh_Y(m-j)[i])\cong \mathrm{Hom}(\oh_Y,\mathcal{I}_V(j-2g+2)[2g-1-i])^{\vee}$. As $1\leq j \leq m\leq 2g-3$, we have $\oh_Y(j-2g+2)\in \oh_Y^{\perp}$. Consider the long exact sequence induced by 
    $$0\rightarrow \mathcal{I}_V(j-2g+2) \rightarrow \oh_Y(j-2g+2) \rightarrow \oh_V(j-2g+2) \rightarrow 0,$$
    we have $\mathrm{Hom}(\oh_Y,\mathcal{I}_V(j-2g+2)[2g-1-i])\cong H^{2g-2-i}(V,\oh_V(j-2g+2))$, the latter cohomology group is zero by above discussions. 
\end{proof}
\begin{corollary}\label{cor of compare}
    The vector bundle $\mathcal{F}_{m,V}$ is simple and the map $\alpha^l_m:V \mapsto [\mathcal{F}_{m,V}]$ is injective. As notations in lemma \ref{compare hom}, $V_2\subseteq V_1$ if and only if $\mathrm{Hom}(\mathcal{F}_{m,V_1},\mathcal{F}_{m,V_2})\ne 0$. 
\end{corollary}

\subsection{The dual of the vector bundles}
In this subsection, we give a explicit description of the dual $\mathcal{F}_{m,V}^{\vee}$. As a consequence, when $\dim(V)=g-2$, we show the rank $2$ bundle $\mathcal{F}_{m,V}$ is of fixed determinant. Recall that the hyperelliptic curve $C$ is the fine moduli space of the spinor bundles. Let $S$ be the universal family on $C\times Y$, which is unique up to a twist by the pull back of a line bundle on $C$. Let $\tau:C\rightarrow C$ be the hyperelliptic involution and denote $\tau':=(\tau,id_Y):C\times Y\rightarrow C\times Y$.
\begin{lemma}\label{the dual of the universal spinor}
    If $g$ is even, we have an isomorphism $S^{\vee}\otimes p^*_C M_S\cong S\otimes p_Y^*\oh_Y(1)$ for some line bundle $M_S$ on $C$. If $g$ is odd, we have an isomorphism $S^{\vee}\otimes p^*_C M_S\cong \tau'^*S\otimes p_Y^*\oh_Y(1)$ for some line bundle $M_S$ on $C$.
\end{lemma}
\begin{proof}
    Let $x\in C$ be a closed point and denote the spinor bundle corresponding to $x$ by $S_x$. By \cite[Theorem 2.8 (ii)]{ottaviani1988spinor}, the restriction of the vector bundle $S^{\vee}\otimes p_Y^*\oh_Y(-1)$ to $\{x\}\times Y$ is isomorphic to $S_x$ when $g$ is even, and is isomorphic to $S_{\tau(x)}$ when $g$ is odd. Since $S$ is the universal spinor bundle, we have $S^{\vee}\otimes p_Y^*\oh_Y(-1)\otimes p^*_C M_S\cong S$ for some line bundle $M_S$ on $C$ when $g$ is even and $S^{\vee}\otimes p_Y^*\oh_Y(-1)\otimes p^*_C M_S\cong \tau'^*S$ for some line bundle $M_S$ on $C$ when $g$ is odd.
\end{proof}

\begin{proposition}\label{the dual of the vector bundes}
    Let $V$ be a linear subspace of dimension $l$ in $Y$ and let $M_S$ be the line bundle on $C$ in Lemma \ref{the dual of the universal spinor}. If $g$ is even, we have an isomorphism $\mathcal{F}_{m,V}^{\vee}\cong \mathcal{F}_{4g-6-l-m,V}\otimes M_S$. If $g$ is odd, we have an isomorphism $\mathcal{F}_{m,V}^{\vee}\cong \tau^*\mathcal{F}_{4g-6-l-m,V}\otimes M_S$
\end{proposition}
\begin{proof}
    Since $\Phi$ is a Fourier mukai transform with integral kernel $S$, the left adjoint $\Phi^*$ is the Fourier mukai transform $R{p_C}_*(S^{\vee}\otimes p_Y^*\omega_Y\otimes p_Y^*(-))[2g-1]$. We have
    $$\mathcal{F}_{m,V}=\Phi^*(\oh_V(m))[-m-2]=R{p_C}_*(S^{\vee}\otimes p_Y^*\omega_Y\otimes p_Y^*\oh_V(m))[2g-3-m].$$
    The dual of it is
    \begin{equation*}
        \begin{aligned}
            &R\mathcal{H}om(\mathcal{F}_{m,V}, \oh_C)=R\mathcal{H}om(R{p_C}_*(S^{\vee}\otimes p_Y^*\omega_Y\otimes p_Y^*\oh_V(m))[2g-3-m], \oh_C)\\
            \cong&R{p_C}_*R\mathcal{H}om(S^{\vee}\otimes p_Y^*\omega_Y\otimes p_Y^*\oh_V(m)[2g-3-m],p_C^!\oh_C)~(\text{Grothendieck-Verdier duality})\\
            \cong&R{p_C}_*R\mathcal{H}om(S^{\vee}\otimes p_Y^*\omega_Y\otimes p_Y^*\oh_V(m)[2g-3-m],p_C^*\oh_C\otimes p_Y^*\omega_Y[2g-1])\\
            \cong&R{p_C}_*(R\mathcal{H}om(p_Y^*\oh_V(m),p_Y^*\oh_Y)\otimes S)[m+2]~(p_C^*\oh_C=\oh_{C\times Y}=p_Y^*\oh_Y)\\
            \cong& R{p_C}_*(S\otimes p_Y^*\oh_V(2g-3-l-m))[m+l-2g+3]~(\oh_V^{\vee}\cong \oh_V(2g-3-l)[l-2g+1]).
        \end{aligned}
    \end{equation*}
    If $g$ is even, by Lemma \ref{the dual of the universal spinor} we have
    \begin{equation*}
        \begin{aligned}
            \mathcal{F}_{m,V}^{\vee}\cong& R{p_C}_*(S^{\vee} \otimes p_Y^*(\oh_V(2g-4-l-m))[m+l-2g+3]\otimes M_S\\
            \cong & R{p_C}_*(S^{\vee} \otimes p_Y^*\omega_Y \otimes p_Y^*(\oh_V(4g-6-l-m))[m+l-2g+3]\otimes M_S\\
            \cong& \mathcal{F}_{4g-6-l-m,V}\otimes M_S.
        \end{aligned}
    \end{equation*}
    If $g$ is odd, note the isomorphisms $\tau'^*p_Y^*\cong p_Y^*$ and $\tau^*\cong \tau_*,\tau'^*\cong \tau'_*$ (The latter is due to that $\tau$ and $\tau'$ are involutions), by Lemma \ref{the dual of the universal spinor} we have 
    \begin{equation*}
        \begin{aligned}
            \mathcal{F}_{m,V}^{\vee}\cong& R{p_C}_*(\tau'^*S^{\vee} \otimes p_Y^*(\oh_V(2g-4-l-m))[m+l-2g+3]\otimes M_S\\
            \cong & R{p_C}_*(\tau'^*S^{\vee} \otimes \tau'^*p_Y^*\omega_Y \otimes \tau'^*p_Y^*(\oh_V(4g-6-l-m))[m+l-2g+3]\otimes M_S\\
            \cong & R{p_C}_*\tau'_{*}(S^{\vee} \otimes p_Y^*\omega_Y \otimes ^*p_Y^*(\oh_V(4g-6-l-m))[m+l-2g+3]\otimes M_S\\
            \cong & \tau_*R{p_C}_*(S^{\vee} \otimes p_Y^*\omega_Y \otimes ^*p_Y^*(\oh_V(4g-6-l-m))[m+l-2g+3]\otimes M_S\\
            \cong &\tau_*\mathcal{F}_{4g-6-l,V}\otimes M_S\cong \tau^*\mathcal{F}_{4g-6-l,V}\otimes M_S.
        \end{aligned}
    \end{equation*}
\end{proof}
\begin{corollary}\label{fix determinant}
    For $g-1\leq m \leq 2g-3$, the map $\det \circ~ \alpha^{g-2}_m$ is constant. That is, the determinant of the rank 2 vector bundle $\mathcal{F}_{m,V}$ does not depend on the dimension $g-2$ space $V$.
\end{corollary}
\begin{proof}
    By Proposition \ref{auto fr m to m+1}, it suffices to prove the assertion for one $m$. 
    
    If $g=2n$ is even, set $m=3n-2$. Apply Proposition \ref{the dual of the vector bundes} in case $l=g-2$, we have $\mathcal{F}_{m,V}^{\vee}\cong \mathcal{F}_{3g-4-m,V}\otimes M_S\cong \mathcal{F}_{m,V}\otimes M_S$. Taking determinant, we get $(\det(\mathcal{F}_{m,V})\otimes M_S)^{\otimes 2}\cong \oh_C$. The map $\mathcal{F}_{m,V}\mapsto \det(\mathcal{F}_{m,V}\otimes M_S)$ induces a morphism from $\mathcal{H}_{g-2}$ to $\mathrm{Pic}(C)$ and the image lies on the set of torsion points of order $2$, which is finite. We conclude by the connectedness of $\mathcal{H}_{g-2}$ that the above morphism is constant, which means $\det(\mathcal{F}_{m,V})$ is fixed.

    If $g=2n+1$ is odd, set $m=3n-1$.  By Proposition \ref{auto fr m to m+1} and Proposition \ref{the dual of the vector bundes}, we have $\mathcal{F}_{m,V}^{\vee}\cong\tau^*\mathcal{F}_{3g-4-m,V}\otimes M_S \cong \tau^*\mathcal{F}_{m+1,V}\otimes M_S\cong \mathcal{F}_{m,V}\otimes \tau^*\mathcal{L} \otimes M_S$, where $\mathcal{L}$ is the line bundle in Proposition \ref{auto fr m to m+1}. We can prove by similar arguments as above. 
\end{proof}

\section{Projections as morphisms}
\label{section_projection_induce_morphism}
 Let $V_l$ denote a linear subspace of dimension $l$. The degree $\deg(\mathcal{F}_{m,V_l})$ does not depend on the choice of $V_l$, so we may set $d_{m,l}=\deg(\mathcal{F}_{m,V_l})$. The difference of the degrees $d_{m+1,l}-d_{m,l}$ is equal to $-2^{g-1-l}$ by Proposition \ref{auto fr m to m+1} and we also have $d_{m,l}=d_{m,l+1}+d_{m-1,l+1}$ by Proposition \ref{ext prop}. By twisting the universal spinor bundle, we may assume that $d_{g-2,g-1}=0$. 
\begin{ex}
    Let $g=3$. The subspace of maximal dimension is a plane $V$ in Y. Under the assumption $d_{1,V}=\deg(\mathcal{F}_{2,V})=0$, we list the associated bundles with their ranks and degrees under the semi-orthgonal decomposition as follows.
\begin{center}
\begin{tabular}{c c c c c }\label{table}
$D^b(Y)=\langle D^b(C),$ &$ \oh_Y,$ & $\oh_Y(1),$ & $\oh_Y(2),$ & $\oh_Y(3),\rangle$ \\ 
\hline
For a plane $V$ &  & $\mathcal{F}_{1,V}$ & $\mathcal{F}_{2,V}$ & $\mathcal{F}_{3,V}$ \\ 
$(rk(\mathcal{F}),\deg(\mathcal{F}))$& & $(1,0)$ & $(1,-1)$ & $(1,-2)$\\
\hline
For a line $L$ &  &  & $\mathcal{F}_{2,L}$ & $\mathcal{F}_{3,L}$ \\ 
& & & $(2,-1)$ & $(2,-3)$\\
\hline
For a point $p$ &  &  &  & $\mathcal{F}_{3,p}$ \\ 
& & &  & $(4,-4)$\\
\end{tabular}
\end{center}
\end{ex}

We're going to show the map $\alpha^l_m$ defines a morphism when the associated bundle is stable.
\begin{proposition}\label{stable then closed embedding}
    Fix an integer $m$ satisfying $2g-3-l\leq m \leq 2g-3$. If the vector bundle $\mathcal{F}_{m,V}$ is stable for each linear subspace $V$ of dimension $l$, then we have a closed embedding $\alpha^l_m:\mathcal{H}_l\rightarrow U_C^s(2^{g-1-l},d_{m,l})$, which is the map defined in Definition \ref{proj of lin} at the level of closed points.
\end{proposition}
\begin{proof}
    Assume that $\mathcal{F}_{m,V}$ is stable. Since the functor $\Phi^*(-\otimes \oh_Y(m))[-m-2]:D^b(Y)\rightarrow D^b(C)$ is of Fourier-mukai type, let $S_m\in D^b(C\times Y)$ be the corresponding integral kernel. Let $\tilde{F}$ be the universal family on $Y\times\mathcal{H}_l$. We define
    $\phi':=\phi_{S_m}\times id_{\mathcal{H}_l}=\phi_{S_m\boxtimes \oh_{\Delta \mathcal{H}_l}}:D^b(Y\times \mathcal{H}_l)\rightarrow D^b(C\times \mathcal{H}_l)$. 
    
    Let $[V]$ denote the closed point corresponding to the quotient $[\oh_Y\rightarrow \oh_V]$ at $\mathcal{H}_{l}$ and we denote $i_{V}:[V]\times C\rightarrow \mathcal{H}_l\times C$. Let $i:C\times[V]\times Y\times \mathcal{H}_l\rightarrow C\times \mathcal{H}_l\times Y\times \mathcal{H}_l$ be the embedding induced by $i_V$. We have $\phi_{S_m}(\oh_V)=\mathcal{F}_{m,V}$ and
    $i^*_V(\phi_{S_m\boxtimes\oh_{\Delta\mathcal{H}_l}}(\tilde{F}))\cong \phi_{i^*(S_m\boxtimes{\oh_{\Delta\mathcal{H}_l}})}(\tilde{F})\cong \phi_{S_m}(\oh_V)=\mathcal{F}_{m,V}$ by the base change theorem and the fact that $\mathcal{H}_l$ is smooth (See \cite[Theorem 2.6]{reid1972complete}). Therefore $\phi'(\tilde{F})$ is a vector bundle on $C\times \mathcal{H}_l$ and induces a morphism to $U_C^s(2^{g-1-l},d_{m,l})$, which is the map $\alpha^l_m$ defined in \ref{proj of lin} at the level of closed points.

    Since $\mathcal{H}_l$ is projective and $U^s_C(2^{g-1-l},d_{m,l})$ is separated, $\alpha_m^l$ is a projective morphism. By Lemma \ref{compare hom}, the morphism $\alpha_m^l$ and the tangent map $d\alpha^l_m:\mathrm{Hom}(\mathcal{I}_V,\oh_V)\rightarrow \mathrm{Ext}^1(\mathcal{F}_{m,V},\mathcal{F}_{m,V})$ are injective on closed points, which implies that $\alpha_m^l$ is a closed immersion.
\end{proof}

We now give an alternative proof of \cite[Theorem 4.8]{reid1972complete} and \cite[Theorem 2]{desale1976classification} via techniques of derived category. 

\begin{theorem}\label{pl iso to pic}
    Let $\mathcal{H}_{g-1}$ be the Hilbert scheme of linear subspaces of maximal dimension $g-1$ in $Y$, then the morphism $$\alpha_m^{g-1}:\mathcal{H}_{g-1}\rightarrow \mathrm{Pic}^{g-2-m}(Y),V \mapsto \mathcal{F}_{m,V}~(g-2\leq m \leq 2g-3)$$ is an isomorphism. 
\end{theorem}
\begin{proof}
    By Proposition \ref{stable then closed embedding}, $\alpha^{g-1}_m$ is a closed immersion. The Fano scheme $\mathcal{H}_{g-1}$ is of the same dimension $g$ as that of $\mathrm{Pic}^{g-2-m}(Y)$ by \cite[Theorem 2.6]{reid1972complete}. Since $\mathrm{Pic}^{g-2-m}(Y)$ is an irreducible variety, the closed embedding $\alpha_m^{g-1}$ is an isomorphism.
\end{proof}

\begin{theorem}\label{stab bd for line}
    Let $L$ be a linear subspace of dimension $g-2$, then the rank 2 bundle $\mathcal{F}_{m,L}$ is stable. In particular, we have an isomorphism $\alpha^{g-2}_m:\mathcal{H}_{g-2}\rightarrow SU_C(2,h_m)$ for $g-1\leq m \leq 2g-3$, where $h_m=\det(\mathcal{F}_{m,L})$ is a fixed line bundle on $C$ of odd degree $2g-3-2m$.
\end{theorem}
\begin{proof}
    Since the degree of line bundle $d_{m,g-1}$ is $g-2-m$, we have $d_{m,g-2}=d_{m,g-1}+d_{m-1,g-1}=2g-3-2m$. We first prove $\mathcal{F}_{m,L}$ is stable and it suffices to prove that for $\mathcal{F}_{g-1,L}$ by Proposition \ref{auto fr m to m+1}. If there exists a subspace $V$ of dimension $g-1$ containing $L$, then by Proposition \ref{ext prop}, as the non-trivial extension of line bundles of degrees differ by 1, $\mathcal{F}_{g-1,L}$ is stable.

    The rank 2 bundle $\mathcal{F}_{g-1,L}$ is of degree $-1$. If $\mathcal{F}_{g-1,L}$ is not stable, let $E\subset \mathcal{F}_{g-1,L}$ be a sub line bundle with $\mathrm{deg}(E)\geq 0$, we may find a suitable line bundle  $E'\subseteq E$ with $\mathrm{deg}(E')=-1$. By Theorem \ref{pl iso to pic}, $E'=\mathcal{F}_{g-1,V}$ for some subspace $V$ of dimension $g-1$. In particular, $\mathrm{Hom}(\mathcal{F}_{g-1,V},\mathcal{F}_{g-1,L})\ne 0$, which implies $L\subset V$ by Corollary \ref{cor of compare}.  Then by the above argument we know $\mathcal{F}_{2,L}$ is stable, which is a contradiction. So the bundles $\mathcal{F}_{m,L}$ are all stable.

    The determinant of $\mathcal{F}_{m,L}$ is fixed by Corollary \ref{fix determinant}, so $\alpha^{g-2}_m$ in fact induce a closed embedding into $SU_C(2,h_m)$. Since $\mathcal{H}_{g-2}$ is of the same dimension $3(g-1)$ as that of $SU_C(2,h_m)$ (\cite[Theorem 2.6]{reid1972complete}), we conclude by the irreducibility of $SU_C(2,h_m)$ that $\alpha_m^{g-2}$ is an isomorphism.
\end{proof}

\subsection{The case $g=3$}
Let $g$ be $3$, $Y$ is a smooth intersection of quadrics in $\mathbb{P}^7$. As in table \ref{table}, we have seen that $\mathcal{F}_{m,V}$ and $\mathcal{F}_{m,L}$ are stable. To prove that $\mathcal{F}_{3,p}$ is also a stable vector bundle of rank $4$ when $g=3$, we begin with a description of the VMRT of $Y$.
\begin{lemma}\label{VMRT of Y}
For each closed point $p\in Y$, the VMRT of $Y$ at $p$ is isomorphic to the intersection of two quadrics in $\mathbb{P}^4$.
\end{lemma}
\begin{proof}
   Let $Y=Q_1\cap Q_2\subset \mathbb{P}^7$, where $Q_i(i=1,2)$ is a quadric in $\mathbb{P}^7$. If we identify the variety of lines passing through $p$ in $\mathbb{P}^7$ as $\mathbb{P}^6$, then the VMRT of $Q_i$ is isomorphic to a 4-dimensional quadric ${Q_i}'\subset H_i \subset \mathbb{P}^6$, where $H_i$ is the variety of lines passing through $p$ contained in the tangent space of $Q_i$ at $p$, which is isomorphic to a hyperplane in $\mathbb{P}^6$. Since the intersection of $Q_i$ is transversely, $H_1\cap H_2$ is isomorphic to $\mathbb{P}^4$. The VMRT of $Y$ at $p$ is isomorphic to the intersection of the two quadrics ${Q_i}'\cap H_1\cap H_2(i=1,2)$ in $H_1\cap H_2$.
\end{proof}

\begin{proposition}\label{g3 and rk4 stable}
    For any point $p\in Y$, $\mathcal{F}_{3,p}$ is stable. 
\end{proposition}
\begin{proof}
For any point $p\in Y$ and any line $L$ passes through it, we have a non-trivial extension by Proposition \ref{ext prop} 
$$0\rightarrow \mathcal{F}_{3,L}\rightarrow \mathcal{F}_{3,p}\rightarrow \mathcal{F}_{2,L}\rightarrow 0,$$
from which we calculate that $r(\mathcal{F}_{3,p})=4,\mathrm{deg}(\mathcal{F}_{3,p})=-4$.

If $\mathcal{F}_{3,p}$ is not stable, let $0\rightarrow E \rightarrow \mathcal{F}_{3,p} \rightarrow G \rightarrow 0$ be a destablizing exact sequence, we have the following possibilities:

\textbf{Case I}: $r(E)=1$ and $\mathrm{deg}(E)\geq -1$.

For the first case, the composite $E\rightarrow \mathcal{F}_{3,p} \rightarrow \mathcal{F}_{2,L}$ does not vanish, otherwise we would have an inclusion $E\hookrightarrow \mathcal{F}_{3,L}$, contradicting to $\mu(E)>\mu(\mathcal{F}_{3,L})$ and the stability of $\mathcal{F}_{3,L}$. Hence we get an inclusion $E\hookrightarrow \mathcal{F}_{2,L}$ and $\deg(E)=-1$ due to the stability of $\mathcal{F}_{2,L}$. By Theorem \ref{pl iso to pic}, $E$ is isomorphic to $\mathcal{F}_{2,V}$ for some plane $V$. From Lemma \ref{VMRT of Y}, the VMRT at $p$ is of dimension at least $2$, we can choose a line $L$ passing through $p$ that is not contained in $V$.  However, by Corollary \ref{cor of compare} $\mathrm{Hom}(E,\mathcal{F}_{2,L})$ vanishes, which contradicts $E\hookrightarrow \mathcal{F}_{2,L}$.

\textbf{Case II}: $r(E)=3$ and $\mathrm{deg}(E)\geq -3$.

Let $E_2$ be the image of $E$ via the morphism $\mathcal{F}_{3,p}\rightarrow \mathcal{F}_{2,L}$, and $E_1$ be the sheaf in the exact sequence $0\rightarrow E_1 \rightarrow E \rightarrow E_2 \rightarrow 0$. Let $G_i$ be the quotient $\mathcal{F}_{4-i,L}/E_i$, we have the diagram
 \[
\begin{tikzcd}
  0 \arrow[r] & E_1 \arrow[hookrightarrow]{d} \arrow[hookrightarrow]{r} & E\ar[hookrightarrow]{d} \arrow[r] & E_2\ar[hookrightarrow]{d} \arrow[r] & 0 \\
  0 \arrow[r] & \mathcal{F}_{3,L} \arrow[r]\ar[d] & \mathcal{F}_{3,p} \arrow[r]\ar[d] & \mathcal{F}_{2,L} \ar[r]\ar[d] & 0\\
  0 \ar[r] & G_1 \ar[r] &G  \ar[r]& G_2 \ar[r] &0.
\end{tikzcd}
\] 
If $r(E_1)=2,~r(E_1)=1$, we have $\deg(E_1)\leq -3,\deg(E_2)\leq -1$ by the stability of $\mathcal{F}_{2,L},\mathcal{F}_{3,L}$.Then we have $\deg(E)=\deg(E_1)+\deg(E_2)\leq -4$, which is a contradiction. If $r(E_1)=1,~r(E_2)=2$, the only possibility is $\deg(E_1)=-2,~\deg(E_2)=-1$. Then $G_2=0$ and $G_1\cong G$. Since the determinant of $\mathcal{F}_{3,L}$ is fixed (Corollary \ref{fix determinant}), $E_1\cong \det(\mathcal{F}_{3,L})\otimes \det(G)^{-1}$ is a fixed line bundle of degree $-2$, which is isomorphic to $\mathcal{F}_{3,V}$ for some fixed plane $V$ by Theorem \ref{pl iso to pic}. We can choose a line $L$ passing through $p$ that is not contained in $V$, contradicting $E_1\hookrightarrow \mathcal{F}_{3,L}$ as above. 

\textbf{Case III}: $r(E)=2$ and $\mathrm{deg}(E)\geq -2$.

Since $\mathcal{F}_{3,p}$ has no destablizing subobject of rank $1$ by \textbf{Case I}, $E$ is a stable vector bundle. The composite $E\rightarrow \mathcal{F}_{3,p} \rightarrow \mathcal{F}_{2,L}$ does not vanish as $\mu(E)>\mu(\mathcal{F}_{3,L})$, so it is an inclusion. If $\deg(E)$ is $-1$, the composite will be an isomorphism, which is impossible since the extension \ref{ext prop} is non-trivial. Therefore, the degree $\deg(E)$ is equal to $-2$. The sheaf $G$ is also a stable bundle: if there exists a destabilizing quotient $G\rightarrow G'$, let $E'$ be the kernel of $\mathcal{F}_{3,p} \rightarrow G'$. Either $r(E')=3,~\deg(E')\geq -3$ or $r(E')=2,~\deg(E')\geq -1$, we reduce to the above cases. 

Now for each line $L$ containing $p$, the composite $E\hookrightarrow \mathcal{F}_{2,L}$ is an inclusion and we have the following diagram
 \[
\begin{tikzcd}
  0 \arrow[r] & 0 \arrow[rightarrow]{d} \arrow[r] & E\ar[d]\arrow[r] & E\ar[hookrightarrow]{d} \arrow[r] & 0 \\
  0 \arrow[r] & \mathcal{F}_{3,L} \arrow[d]\arrow[r] & \mathcal{F}_{3,p} \ar[d]\arrow[r] & \mathcal{F}_{2,L} \ar[r]\ar[d] & 0\\
  0 \ar[r] &\mathcal{F}_{3,L} \ar[r] &G \ar[r] &\oh_{x} \ar[r] &0,
\end{tikzcd}
\]
The quotient of $G$ by $\mathcal{F}_{3,L}$ is a skyscraper sheaf $\oh_x$ for some closed point $x\in C$. Since $\det(\mathcal{F}_{3,L})$ is fixed by Corollary \ref{fix determinant} and $\oh_C(x)$ is isomorphic to $ \det(G)\otimes \det(\mathcal{F}_{3,L})^{-1}$, the point $x$ is fixed, that is, it does not depend on the choice of $L$. Let's denote the VMRT of $Y$ at $p$ by $M$ and denote the projective bundle $\pi:\mathbb{P}_C(G)\rightarrow C$. We claim that 
\begin{claim}\label{claim to contradiction}
    The above diagram will induce a morphism $\psi:\mathbb{P}_C(G)\rightarrow U_C^s(2,-3)$. If we view $M$ as a closed subvariety of $\mathcal{H}_1$, we have $\alpha_3^1(M)\subseteq \psi(\pi^{-1}(x))$.
\end{claim}
The claim contradicts that $\dim(M)$ is at least $2$, so the \textbf{Case III} is ruled out.

\begin{proof}[Proof of the claim]
    We first construct a morphism $\psi:\mathbb{P}_C(G)\rightarrow U_C^s(2,-3)$ by elementary transformation. For each point $l\in \mathbb{P}_C(G)$ lying above $c\in C$, we have the elementary transformation $E_{l}G$ defined by 
    $0\rightarrow E_lG \rightarrow G \rightarrow G_c/l \otimes \oh_c \rightarrow 0$. 
    As $G$ is a rank $2$ stable bundle of degree $-2$, every line bundle contained in $E_lG$ is of degree less than $-1$, which means $E_{l}G$ is again stable of degree $-3$. Consider $\mathbb{P}_C(G)\times C$ and we denote
    $C\xleftarrow{\pi}\mathbb{P}_C(G)\xleftarrow{pr_1}\mathbb{P}_C(G)\times C \xrightarrow{pr_2}C$. 
    The morphism $\gamma:=(id_{\mathbb{P}_C(G)},\pi):\mathbb{P}_C(G)\rightarrow \Gamma \subset \mathbb{P}_C(G)\times C$ embeds $\mathbb{P}_C(G)$ as a divisor $\Gamma$ in $\mathbb{P}_C(G)\times C$, we get a canonical surjection
    $$pr_2^*G\rightarrow pr_2^*G\otimes \oh_{\Gamma}\cong pr_1^{*}\pi^*G\otimes \oh_{\Gamma} \rightarrow pr_1^*(\oh_{\mathbb{P}_C(G)}(1))\otimes \oh_{\Gamma}$$
    and the kernel of it restricts to $E_lG$ on $\{l\}\times C $, inducing a morphism $\psi:\mathbb{P}_C(G)\rightarrow U^s_C(2,-3)$. 

    By Theorem \ref{stab bd for line}, we have an embedding of the VMRT $M$ of $Y$ at $p$
    $$\alpha_3^{1}|_M:M\rightarrow U_C^s(2,-3),~L\mapsto \mathcal{F}_{3,L}.$$
    The diagram means that every $\mathcal{F}_{3,L}$ comes from an elementary transformation over the point $x\in C$, i.e. $\alpha_3^1(M)\subseteq \psi(\pi^{-1}(x))$.
\end{proof}

\end{proof}

\begin{theorem}\label{Y as locus in rk4 moduli}
    Let $Y$ be a smooth intersection of quadrics in $\mathbb{P}^7$, we have a closed embedding
    $\alpha^0:Y\rightarrow SU^s_C(4,h),~p\mapsto \mathcal{F}_{3,p}$, where $h=\det(\mathcal{F}_{3,p})$ is a fixed line bundle of degree $-4$.
\end{theorem}
\begin{proof}
    By Proposition \ref{ext prop} and Corollary \ref{fix determinant}, we see that $\mathcal{F}_{3,p}$ is also of fixed determinant. Then we apply Proposition \ref{stable then closed embedding} and Proposition \ref{g3 and rk4 stable}.
\end{proof}

\section{Further questions: stability of projection and Brill-Noether conditions}
\label{section_seek_BN_conditions}
If $Y:=Q_1\cap Q_2$ is a smooth del Pezzo threefold of degree $4$, then the associate hyperelliptic curve $C$ is of genus two and $Y\cong\{F\in\mathcal{M}_C(2,1):\mathrm{hom}(F,\mathcal{R})\geq 1\}$, where $\mathcal{L}$ is a degree one line bundle over $C$ and 
$\mathcal{R}=\Phi^{!}(\oh_{Y_4})[-1]$ is a second Raynaud bundle on the hyperelliptic curve $C$. To show this, first we prove stability of projection of skyscraper sheaf of each point of $Y$ into the Kuznetsov component $\Ku(Y)\simeq D^b(C)$. Then we can show projection of skyscraper sheaf of each point of $Y$ satisfies the \emph{Brill-Noether} condition: $\mathrm{Hom}(\mathrm{pr}(\oh_y),\Phi^{!}(\oh_Y))=k^5$. Finally, we make use of results in \cite{rota} to conclude (see \cite[Section 6.1]{feyzbakhsh23torelli}).  Then the following question is very natural for general case. 

\begin{quest}
    Is $\mathcal{F}_{m,V}$ a stable bundle for each $l$ dimensional subspace $V$ and each integer $m(2g-3-l\leq m \leq 2g-3)$? If the answer is affirmative, can we describe the Brill-Noether conditions of $\alpha^l_m(\mathcal{H}_l)$ in $U^s_C(2^{g-1-l},d_{m,l})$?
\end{quest}
We have seen that $\mathcal{F}_{3,p}$ is a stable rank $4$ bundle for each point $p\in Y$ when $g=3$. One can show the rank $4$ bundle $\mathcal{F}_{m,P}$ is semistable for a $g-3$ dimensional subspace $P$ if there's a $g-2$ subspace $L$ with $P\subset L$ when $g>3$.
\begin{proposition}\label{containing in big space then rk4 stable}
    Assume $g\geq 3$. Let $P$ be a $g-3$ dimensional subspace in $Y$. If there's a $g-2$ dimensional subspace with $P\subset L$, then $\mathcal{F}_{m,P}$ is semistable for $g\leq m \leq 2g-3$.
\end{proposition}
\begin{proof}
    By Proposition \ref{auto fr m to m+1}, we only need to show $\mathcal{F}_{g,P}$ is semistable and we prove by contradiction. Let $E\hookrightarrow \mathcal{F}_{g,P}$ be a subbundle with $\mu(E)>\mu(\mathcal{F}_{g,P})$ and let $E_1,E_2$ be the sheaves in the diagram
    \[
    \begin{tikzcd}
        0 \arrow[r] & E_1 \arrow[hookrightarrow]{d} \arrow[hookrightarrow]{r} & E\ar[hookrightarrow]{d} \arrow[r] & E_2\ar[hookrightarrow]{d} \arrow[r] & 0 \\
  0 \arrow[r] & \mathcal{F}_{g,L} \arrow[r] & \mathcal{F}_{g,P} \arrow[r] & \mathcal{F}_{g-1,L} \ar[r] & 0.
    \end{tikzcd}
    \]
    Under the assumption $d_{g-2,g-1}=0$, the ranks and degrees of $\mathcal{F}_{g,L},\mathcal{F}_{g,P},\mathcal{F}_{g-1,L}$ are the same as that of $\mathcal{F}_{3,L},\mathcal{F}_{3,p},\mathcal{F}_{2,L}$ respectively as in \ref{table}. By Theorem \ref{stab bd for line}, we have $\mu(E_1)\leq \mu(\mathcal{F}_{g,L})=-\frac{3}{2},~ \mu(E_2)\leq \mu(\mathcal{F}_{g-1,L})=-\frac{1}{2}$ and $-1<\mu(E)<0$.

    If $r(E)=1$, then $\deg(E)\leq -1=\mu(\mathcal{F}_{g,P})$, which is impossible. If $r(E)=2$, then the only possibility is that $E_1=0$ and $E\cong E_2\cong \mathcal{F}_{g-1,L}$, contradicting to the fact that the extension \ref{ext prop} is non-trivial. If $r(E)=3$, either $r(E_1)=1,r(E_2)=2$ or $r(E_1)=2,r(E_1)=1$. We have $\deg(E_1)\leq -2,\deg(E_2)\leq -1$ in the former case and $\deg(E_1)\leq -3,\deg(E_2)\leq -1$ in the latter case. Then $\deg(E)=\deg(E_1)+\deg(E_2)\leq -3$ in both cases and we get a contradiction.
\end{proof}

Next we make an attempt to seek for Brill-Noether conditions for $\mathcal{F}_{m,V}$. For general case, we consider $\mathcal{R}_d:=\Phi^!(\oh_Y(d))[-1],0\leq d\leq 2g-3$. 
\begin{proposition}\label{prop_BN_condition}
    For $0\leq d\leq 2g-3$, $\dim(V)=l$, we have
    $$
    \mathrm{ext}^1(\mathcal{F}_{m,V},\mathcal{R}_d)=\begin{cases}
        \dbinom{l+d}{d}&,m=2g-3-l\\
       0&, \text{~otherwise}
    \end{cases}
    .$$
\end{proposition}
\begin{proof}
    By definition, we have
    \begin{equation*}
        \begin{aligned}
            \mathrm{Ext}^1(\mathcal{F}_{m,V},R_d)&\cong\mathrm{Hom}(\Phi^{*}(\oh_V(m))[-m-2],\Phi^!(\oh_Y(d)))\\
            &\cong \mathrm{Hom}(\Phi\Phi^*(\oh_V(m)),\oh_Y(d)[m+2]).\\
        \end{aligned}
    \end{equation*}
    If we can show $\mathrm{Hom}(\Phi\Phi^{*}(\oh_V(m)),\oh_Y(d)[m+2])\cong \mathrm{Hom}(\oh_V(m),\oh_Y(d)[m+2])$, then
    \begin{equation*}
        \begin{aligned}
            \mathrm{Ext}^1(\mathcal{F}_{m,V},R_d)&\cong \mathrm{Hom}(\oh_V(m),\oh_Y(d)[m+2])\\
            &\cong \mathrm{Hom}(\oh_Y(d)[m+2],\oh_V(2-2g+m)[2g-1])^{\vee}\\
            &\cong H^{2g-3-m}(V,\oh_V(2-2g+m-d))^{\vee}
        \end{aligned}
    \end{equation*}
    and one can see that the above is $H^{l}(\oh_V(-l-1-d))^{\vee}$ when $m=2g-3-l$ and vanished when $2g-3-l<m\leq 2g-3$.

    To prove $\mathrm{Hom}(\Phi\Phi^{*}(\oh_V(m)),\oh_Y(d)[m+2])\cong \mathrm{Hom}(\oh_V(m),\oh_Y(d)[m+2])$, we use the same argument in Proposition \ref{compare hom}. We set 
    $$L_j=\mathbb{L}_{\oh_Y(m-j)}\cdots \mathbb{L}_{\oh_Y(m)}\oh_{V}(m),~L_{-1}=\oh_{V}(m),~L_m=\Phi\Phi^{*}(\oh_{V}(m)).$$
    where $0\leq j\leq m$. The left mutation is defined by
    $$\bigoplus_{i\leq j}\mathrm{Hom}(\oh_Y(m-j)[i],L_{j-1})\otimes \oh_Y(m-j)[i]\rightarrow L_{j-1} \rightarrow L_j.$$
    We prove $\mathrm{Hom}(\oh_Y(m-j)[i+1],\oh_Y(d)[m+2])=\mathrm{Hom}(\oh_Y(m-j)[i],\oh_Y(d)[m+2])=0$. 
    
    If $m-j>d$, then $\oh_Y(d)\in \oh_Y(m-j)^{\perp}$ and we're done. If $m-j\leq d$, we note that $i+1\leq j+1 \leq m+1$, and the above vanishing is from $H^{n}(\oh_Y(r))=0$ for any $n>0,r\geq 0$.

    Apply $\mathrm{Hom}(-,\oh_Y(d)[m+2])$ to the above exact triangle, we see that 
    $$\mathrm{Hom}(L_{j-1},\oh_Y(d)[m+2])\cong \mathrm{Hom}(L_j,\oh_Y(d)[m+2])$$
    implying $\mathrm{Hom}(\Phi\Phi^{*}(\oh_V(m)),\oh_Y(d)[m+2])\cong \mathrm{Hom}(\oh_V(m),\oh_Y(d)[m+2])$. 
\end{proof}

We expect the image $\mathcal{F}_{m,V}$ of twisted structure sheaf $\oh_V$ of linear subspace $V$ of intersection of two quadrics $Y$ under the projection functor $\Phi^*\cong\mathrm{pr}$ satisfy some of \emph{Brill-Noether} condition in Proposition~\ref{prop_BN_condition}. Thus the projection functor would be expected to induce a closed immersion of $\mathcal{H}_l$ into \emph{Brill-Noether locus} of moduli space $U^s_C(2^{g-1-l},d_{m,l})$. Then it is possible to show that $\mathcal{H}_l$ is actually isomorphic to such Brill-Noether locus as we did for del Pezzo threefold of degree $4$. Nevertheless, it is not clear to us which condition we should choose to reconstruct $\mathcal{H}_l$ inside of $U^s_C(2^{g-1-l},d_{m,l})$. We will study this problem in the updated version of the paper.

\bibliography{ref}

\bibliographystyle{alpha}

\end{document}